\documentclass[a4paper,10pt,intlimits,oneside]{amsart}
\usepackage{amsmath}
\usepackage{amsthm}
\usepackage{latexsym}
\usepackage{amssymb}
\usepackage{xcolor}
\usepackage{dsfont}
\usepackage{mathrsfs}
\usepackage[colorlinks=true]{hyperref}
\numberwithin{equation}{section}
\numberwithin{figure}{section}
\def\expo_#1{{\rm e}^{#1}}

\def\R{{\mathbb R}}
\def\C{{\mathbb C}}

\def\1{1\!{\rm l}}

\def\build#1_#2^#3{\mathrel{\mathop{\kern 0pt#1}\limits_{#2}^{#3}}}
\def\td_#1,#2{\mathrel{\mathop{\build\longrightarrow_{#1\rightarrow #2}^{}}}}

\def \mod{\mathrm{mod}}

\newcommand{\ben}{\begin{equation}}
\newcommand{\een}{\end{equation}}
\newcommand{\beno}{\begin{eqnarray*}}
\newcommand{\eeno}{\end{eqnarray*}}

\newtheorem{theorem}{Theorem}
\newtheorem{corollary}{Corollary}
\newtheorem{proposition}{Proposition}
\newtheorem{lemma}{Lemma}
\newtheorem{remark}{Remark}

\date{}
\title{BGS}
\newcounter{rea}
\newcounter{reb}
\newcounter{res}
\title[Stein's Theorem and weights]{Stein's Theorem  in the upper-half plane and Bergman spaces with weights}
\author{Aline Bonami}
\address{ Institut Denis Poisson, D\'epartement de Math\'ematiques, Universit\'e d'Orl\'eans, 45067Orl\'eans Cedex 2, France}
\email{aline.bonami@gmail.com}
\author{Sandrine Grellier}
\address{ Institut Denis Poisson, D\'epartement de Math\'ematiques, Universit\'e d'Orl\'eans, 45067
Orl\'eans Cedex 2, France}
\email{sandrine grellier@univ-orleans.fr}
\author {Beno\^it Sehba}
\address{Department of Mathematics, University of Ghana, PO. Box LG 62 Legon, Accra, Ghana }
\email{ bfsehba@ug.edu.gh}
\begin{document}
\maketitle

\begin{abstract}
  This is a companion paper to our previous one, {\sl  Avatars of Stein's Theorem in the complex setting.} In this previous paper, we gave a sufficient condition for an integrable function in the upper-half plane to have an integrable Bergman projection. Here we push forward methods and establish in particular a converse statement. This  naturally leads us to study a family of weighted Bergman spaces for logarithmic weights $(1+\ln_+(1/\Im m (z))+\ln_+(|z|))^k,$ which have the same kind of behavior respectively at the boundary and at infinity. We introduce their duals, which are logarithmic Bloch type spaces and interest ourselves in multipliers, pointwise products and Hankel operators. 
  
\end{abstract}

\let\thefootnote\relax
\section{Introduction}
{\em This is a corrected version of our paper \cite{BGS3}, which was unfortunately sent to print with forgotten modifications in Section 4.2.}

This article revisits   previous work of the authors, essentially \cite{BGS1} but also \cite{BGS}. It may serve partly as a survey but contains also  new information as well as some extensions of the results given there. It also serves as an erratum for a missing condition in the study of Hankel operators.

Let us first recall the main results of these two papers.

In \cite{BGS}, we considered an analog for the whole space of Stein's Theorem (\cite{stein}), which says  that, whenever $f$ is an $L^1(\R^n)$  non negative function, which is supported in a ball $B$, its Hardy-Littlewood maximal function $Mf$ is  integrable on $B$ if and only if $f\ln_+(f)$ is integrable. We proved the following, for which \cite{Sola} was also a source of inspiration:
\begin{theorem} \label{main-proc}
A function $f\in L^1(\R^n)$ is in $H^1(\R^n)$ when it satisfies the two conditions:
\begin{equation}\label{mean}
\int_{\R^n} f dx =0,
\end{equation}
\begin{equation}\label{log}
\int_{\R^n} |f|(\ln_+(|f|)+\ln_+(|x|)) dx <\infty.
\end{equation}
Conversely, let $f\in L^1(\R^n)$ be a  function which is  non negative outside a compact set and such that
$$\int|f_-(x)| \ln_+(|f_-(x)|) dx<\infty.$$ 
 If $f$ belongs to $H^1(\R^n)$, 
 then $f$ satisfies Condition \eqref{mean} and \eqref{log}.
 \end{theorem}
Of course it is not possible for a non negative function $f$ to be itself in $H^1(\R^n)$ since all functions in $H^1(\R^n)$ have integral $0$.  The first part of this statement may be found directly in \cite{BGS}, while the converse part looks more general.  In fact, with the assumptions given above, we can write, assuming that $f$ has integral $0,$ that
$$ f=\left(f_+ -\big(\int f_+ dx\big) \chi_B \right)-\left(f_- -\big(\int f_- dx\big)\chi_B \right),$$ with $B$ the unit ball centered at $0$. The second term is in $H^1$ since it has integral $0$, is in $L\log L$ and is compactly supported (just use the first statement of the theorem). The statement given in \cite{BGS} expresses that the first one is in $H^1$ if and only if $f_+$ satisfies \eqref{log}. This concludes the proof of this new formulation.

In \cite{BGS1} we consider an analogous situation, but in the upper-half-space 
$$\mathbb C_+:=\{z\in\C, \; \Im m(z)>0\}$$
endowed with the Lebesgue measure noted
$$dV(z)=dx dy\quad  z=x+iy.$$
We denote by $\mathcal H:=\mathcal H(\C_+)$ the space of holomorphic functions in $\C_+$. We are now  interested in the Bergman projection, that is, the orthogonal projection from $L^2(\C_+)$ onto its subspace of holomorphic functions $A^2(\C_+)$. More generally we define
$$A^p(\C_+):=L^p(\C_+)\cap \mathcal H.$$
We denote by $P$ the Bergman projection. Its kernel, called the Bergman kernel, is given by $$K(z,\zeta)=\frac 1{\pi(z- \overline\zeta)^2}.$$ We denote by $P^+$ the operator with kernel $|K(z, \zeta)|.$ The following statement plays the role of Theorem  \ref{main-proc} in this context. Here the Riesz transforms, that characterize $H^1(\R^n)$, are replaced by the Bergman projection. 

The statement is an improvement of the one of \cite{BGS1}.
\begin{theorem} \label{Bergman}
The Bergman projection of a function $f\in L^1(\C_+)$ (respectively $P^+f$)  is an integrable function  if  it satisfies the two conditions:
\begin{equation}\label{meanB}
\int_{\C_+} f dV(z) =0,
\end{equation}
\begin{equation}\label{logB}
\int_{\C_+} |f(z)|(\ln_+((\Im m z)^{-1})+\ln_+(|z|)) dV(z) <\infty.
\end{equation}
Conversely, let $f\in L^1(\C_+)$ be a non negative function outside a compact set of $\C_+$.  When
 $P^+f$ belongs to $L^1(\C_+)$, then   $f$ satisfies Conditions \eqref{meanB} and \eqref{logB}.
\end{theorem}
The converse part  is  new. 
We postpone details and proofs to the next section. 
The first idea of the proof is the fact that, since {\bf $f$ has integral $0,$ }
we may subtract $K(z, i)$ to the Bergman kernel $K(z, \zeta)$ and write
\begin{eqnarray*}
Pf(z)=P_{\mod}f(z):=\int_{\C_+} K_{\mod} (z, \zeta) f(\zeta) dV(\zeta), \; & K_{\mod}(z, \zeta) = K(z, \zeta)-K(z, i)\\
P^+f(z)=P_{\mod}^+f(z):=\int_{\C_+} K^+_{\mod} (z, \zeta) f(\zeta) dV(\zeta), \; & K^+_{\mod}(z, \zeta) = 
 |K(z, \zeta)|-|K(z, i)|.
\end{eqnarray*}
These new kernels are integrable in $z$ so that $P_{\mod}$ and $P^+_{\mod}$ are defined on the entire space $L^1(\C_+)$.
Let us emphasize that this modification of the Bergman kernel is classical, see for instance \cite{CoifmanRochberg, SH}.

Remark that while $P$ and $P^+$ satisfy the inequality $|Pf|\leq P^+(|f|),$ there is no inequality between modified operators, and $P^+_{\mod}$ is no more non negative.

One may be surprised to see a difference between the statement in $\R^n$ and this one, with $\ln_+ (|f|)$ replaced by $\ln_+((\Im m z)^{-1})$ in the upper-half plane. Indeed, the Bergman projection is a singular integral, like Riesz transforms in $\R^n$. This is clarified by the following proposition, whose proof is given in \cite{BGS1}.

\begin{proposition}
Whenever  $f\in L^1(\C_+)$ is such that
\begin{equation}\label{logWeak} \int_{\C_+} |f(z)|(1+\ln_+(|f(z)|)+\ln_+(|z|))dV(z)<\infty,\end{equation}
then 
$f$ satisfies Condition \eqref{logB}.
\end{proposition}
There exist functions that satisfy Condition \eqref{logB} but not Condition \eqref{logWeak}: indeed any $L^1$ function that is supported in a compact set of $\C_+$ satisfies  \eqref{logB} but does not in general satisfy \eqref{logWeak}. So one finds a strictly weaker condition in the Bergman situation. Moreover we do not need to appeal to a Musielak Orlicz space as in the case of $\R^n$ to have a critical condition for $Pf$ to be in $L^1(\C_+).$

In the second section we will give the proof of Theorem \ref{Bergman} and related statements.
The third section is at the same time a survey and a generalization to   weighted Bergman spaces $A^1_\Omega:= \mathcal H \cap L^1(\Omega dV).$ We will particularly be interested in weights $\omega^k, $ where 
\begin{equation}
 \label{omega}
\omega(z):=1+\ln_+\left(\frac 1{\Im m(z)}\right)+ \ln_+(|z|),\quad z\in\C_+.\end{equation}
 Our  starting point in \cite{BGS1}  was the fact that, when looking at the product of a function in $A^1(\C_+)$  and a function in the Bloch class, that is, an element of its dual, one obtains a function in $A^1_{\omega^{-1}}(\C_+).$ Remark that one may find a better estimate, see \cite{BGS1} and also \cite{BBT} in the case of the unit ball. Indeed, for $f\in A^1(\C_+)$ and $b\in \mathcal B,$ the product $h=fb$ satisfies the inequality
\begin{equation}\label{prodStrong} \int_{\C_+} \frac{|h(z)|}{1+\ln_+(|h(z)|)+\ln_+(|z|)}dV(z)<\infty,\end{equation}
which is stronger as it can be proved using atomic decomposition of both spaces. We could as well define spaces of functions $f$ such that
$$\int_{\C_+} \frac{|h(z)|}{1+(\ln_+(|h(z)|))^k+(\ln_+(|z|))^k}dV(z)<\infty.$$
But spaces $A^1_{\omega^k}(\C_+)$ have the advantage of being Banach spaces while having the same duals, so that we will limit ourselves to them. 

For $\Omega$ a positive function defined on $\mathbb{C}_+$, the Bloch type space $\mathcal{B}_\Omega(\mathbb{C}_+)$ consists of all holomorphic functions $f$ such that
\begin{equation}\label{semi}
\|f\|_{\dot {\mathcal{B}}_{\Omega}}:=\sup_{z\in \mathbb{C}_+}\left(\Im mz\right)\Omega(z)|f'(z)|<\infty.
\end{equation}
Remark  that $\|f\|_{\dot {\mathcal{B}}_{\omega^k}}$ is only a semi-norm. It is a norm on the quotient space  modulo constants 
\begin{equation}\label{dotB}
    \dot{\mathcal{B}}_{\omega^k}(\mathbb{C}_+):= \mathcal{B}_{\omega^k}(\mathbb{C}_+)/\C.
\end{equation} When we deal with Bloch spaces, we introduce the norm
\begin{equation}\label{normB}
\|f\|_{ {\mathcal{B}}_{\Omega}}:=|f(i)|+\sup_{z\in \mathbb{C}_+}\left(\Im mz\right)\Omega(z)|f'(z)|<\infty.
\end{equation}
The necessity to consider at the same time Bloch functions and equivalence classes of such functions modulo constants is linked to the fact that we will consider dual spaces on one hand and pointwise products on the other hand. In particular, we will consider products of functions respectively  in $A^1_{\omega^k}$ and in the Bloch type space $\mathcal{B}_{\omega^{-k}}(\mathbb{C}_+)$. 

We will characterise pointwise multipliers between the spaces $\mathcal{B}_{\omega^k}$.
We will see that the dual space of  $A^1_{\omega^{-k}}$ is the space of classes modulo constants  
$\mathcal{B}_{\omega^k}(\mathbb{C}_+)/\C$. 

Only $\mathcal{B}_{\omega}(\mathbb{C}_+)$ was introduced in \cite{BGS1} under the name $\mathcal B_{\log}$. The behavior of these spaces depends on the value of $k.$ Our  first elementary observation is the fact that {\sl $\mathcal{B}_{\omega^k}(\mathbb{C}_+)$ is an algebra for $k>1.$} 


We end this work by an {\sl erratum} in the paper \cite{BGS1} concerning the boundedness of Hankel operators on $A^1(\C_+)$.
 Recall that the Hankel operator $h_b$ is defined  by $h_b(f)=P\left(f\Bar{b}\right)$ when this makes sense. Equivalently, for $b$ a Bloch function, $h_b$ is defined as an operator from $A^2(\C_+)$ into itself as the operator satisfying $\langle h_b(f),g\rangle =\langle b,fg\rangle $ for any $f,g\in A^2(\C_+)$. Here $\langle \cdot,\cdot\rangle$ stands for the usual Hermitian inner product on $L^2(\C_+)$.\\
 We want to extend this operator as an operator from $A^1(\C_+)$ to itself. It is natural to consider  quantities such as $\langle b,fg\rangle $ for $f\in A^1(\C_+)$ and $g$ an element of the Bloch space. However, as $g$ is defined in the Bloch space up to a constant, in order to give a meaning to this term, one has to assume that $\langle b,f\rangle =0$. We forgot to add this condition in our previous paper \cite{BGS1}. Our corrected statement is the following.


We define the space $A^1_{b}(\C_+)$ as the subspace of functions $f$ in $A^1(\C_+)$ satisfying $\langle b,f \rangle=0$.
\begin{theorem}\label{hankelexp}
Let $b$ be in $\mathcal H(\C_+)$.  Then the Hankel operator $h_b$ extends to a bounded operator from $A^1_{b}(\mathbb{C}_+)$ to $A^1(\mathbb{C}_+)$ if and only if $b$ belongs to $\mathcal{B}_\omega(\mathbb{C}_+)$.
Moreover, when $f\in A^1(\C_+)$ and $\langle b,f\rangle \neq 0$, $h_b(f)\notin L^1(\C_+)$.
\end{theorem}
\medskip

In the following, we will use the notation 
$A\lesssim B$ (respectively $A\gtrsim B$, $A\simeq B$) whenever there exists a uniform constant $C>0$ such that $A\le CB$ (respectively $A\ge CB$, respectively $A\lesssim B$ and $A\gtrsim B$).  When $A\lesssim B$ and $A\gtrsim B$, we write $A\simeq B$.

We will use the notation ``$\log $"  for the holomorphic logarithm, and the notation ``$\ln$" for the real one.

\section{Theorem \ref{Bergman}: Proof and related statements}

We refer to \cite{BGS1} for the proof of the sufficiency of Conditions \eqref{meanB} and \eqref{logB}. We start with a lemma, which completes Lemma 2 of \cite{BGS1}, and which will have as corollary that the theorem is sharp for $P$ and $P^+$. \\
We recall the definition of the weight $\omega.$
\begin{equation}
\omega(z):=1+\ln_+\left(\frac 1{\Im m(z)}\right)+ \ln_+(|z|),\quad z\in\C_+.\end{equation}
\begin{lemma}\label{ineq-fund}
Let $\zeta\in \C_+$. For $|\zeta-i|>1,$, the following inequality holds:
 \begin{equation}
    I(\zeta):= \int_{\C_+} \left||z-\overline \zeta|^{-2} -|z+i|^{-2}\right|dV(z)\simeq \omega(\zeta).
 \end{equation}
\end{lemma}
  \begin{proof}
  The above bound has been proved in \cite{BGS1}. We come back to the proof in order to give the bound below. We write $z=x+iy$ and $\zeta=\xi+i\lambda.$
From the continuity of the integral as a function of $\zeta$, it is bounded below on any compact set of $\C_+$. So it is sufficient to prove the inequality for $\lambda<\varepsilon<1$ and for $|\zeta|>M.$ We have two cases to consider.
 \begin{enumerate}
      \item 
If $\zeta$ is such that $\lambda<1/4$: we  write that $I(\zeta)$ is bounded below by the integral over the domain $E_0:=\{z\in \C_+, |x-\xi|<y+\lambda, y<1/4.\}$ On this set, the inequality  $|z-\overline \zeta|^{2}\leq 2(y+\lambda)^2<1/2$ holds so that
 $$\int_{E_0} |z-\overline \zeta|^{-2} dV(z)\geq \int_0^{1/4} \frac{dy}{y+\lambda}=\ln\left(1+\frac 1{4\lambda}\right),$$
 while  $$\int_{E_0} |z+i|^{-2} dV\leq  1/4.$$
So, for $\lambda<1/4$, we have
  $$I(\zeta)\gtrsim 1+\ln_+(\lambda^{-1}).$$
  
   \item If $|\zeta|>M:$ we  write that $I(\zeta)$ is bounded below by the integral on the domain $E_1:=\{z\in \C_+, |z+i|<\frac 14|\zeta-i|, |z|>1\}$. 
      The function to integrate is bounded below by $c|z+i|^{-2},$ so that
      $$I(\zeta)\gtrsim \int_{E_1}|z|^{-2} dV (z)\gtrsim \ln(\frac 14|\zeta|-2).$$
        Here, we used the fact that the half disc $\{|z|<\frac 14|\zeta|-2, y>0\}$ is contained in $E_1$. So finally, we find that 
       $I(\zeta)\gtrsim 1+ \ln_+(|\zeta|)$
      for $|\zeta|\geq M$ large enough. This completes the proof of the lemma.
  \end{enumerate}
  \end{proof}
  \begin{remark}\label{LemmaKmod}
  The same is valid for the kernel $K_{\mod}(z,\zeta)$ as well.
  \end{remark}
 As a corollary, we find a family of functions $f$ non negative outside a fixed compact set and  for which the $L^1$ norm of $Pf$ is equivalent to the  norm of $f$ in the weighted  space $L^1_\omega(\C_+)$. This is given in the following proposition.

 \begin{proposition}
For each $\zeta=\xi+i\lambda \in \C_+$ such that $|\zeta-i|>1,$ let $$f_\zeta:=\frac{4}{\pi \lambda^2}\chi_{B(\zeta, \lambda/2)}-\frac {4}{\pi}\chi_{B(i, 1/2)}.$$ 
Here $\chi_E$ stands for the characteristic function of $E$.
Then $$\|f_\zeta\|_1=2,\;\text{ while }\;
\|Pf_\zeta\|_1\simeq \omega(\zeta).$$
In particular,
$$\|Pf_\zeta\|_1\simeq \|f_\zeta\|_{L^1_\omega}.$$
Constants may be chosen independent of $\zeta.$
\end{proposition}
\begin{proof}
As a consequence of the mean value property of holomorphic functions $Pf_\zeta (z)=\frac 1 \pi\left[(z-\overline \zeta)^{-2} -(z+i)^{-2}\right].$ The result follows from Lemma \ref{ineq-fund}.  
\end{proof}
This proposition proves that Condition \eqref{logB} is optimal.\\

We have the same proposition with $P^+$ in place of $P$.
\begin{proposition}
We have 
$\|f_\zeta\|_1+\|P^+f_\zeta\|_1\simeq \omega(\zeta).$
Constants may be chosen independent of $\zeta.$
\end{proposition}
\begin{proof}
We use again Remark \ref{LemmaKmod}. We prove that the difference between $P^+f_\zeta$ and $P_{\mod} f_\zeta$ is bounded in $L^1$-norm. We prove that $\frac 4{\pi \lambda^2 }P^+(\chi_{B(\zeta, \lambda/2)})-|K(z, \zeta)|$ has an $L^1$ norm bounded by a uniform constant. As the other  terms in $f_\zeta$ is analogous, it will give the result.
But for $\zeta'\in B(\zeta, \lambda/2),$ the function  $|K(\cdot, \zeta)|-|K(\cdot, \zeta')|$ is bounded, up to a constant,  by $\frac{\lambda}{|z-\overline \zeta|^3}, $. This quantity  is uniformly in $L^1$. It ends the proof.
\end{proof}

Next, we prove that the condition \eqref{meanB} is necessary for $Pf$ to be integrable. 
 \begin{proposition} \label{zeroP}
Let $f\in L^1(\mathbb C_+) $ 
 such that $Pf$ (resp. $P^+f$) is in $L^1(\C_+)$.  Then $f$ has integral $0$.
\end{proposition}
\begin{proof}
We prove it for $Pf$ in $L^1$. The same proof works when $P^+f$ is in $L^1$. We write 
$$z ^2 Pf(z)=c\int_{\C_+} \frac{f(\zeta) dV(\zeta)}{(1-\overline \zeta/z)^2}$$
and we use Lebesgue Dominated Convergence Theorem. When $|x|<y$, as $|1- \overline \zeta/z|>y/|z|>1/2 $, we get for $|z|$ tending to $\infty$ with $|x|<y$
$$z ^2 Pf(z) \longrightarrow \int _{\C_+} f(\zeta) dV(\zeta).$$
Since $z\mapsto z^{-2}$ is not integrable over this domain, this forces the integral to be $0$ if $Pf$ is integrable.
\end{proof}

\begin{remark} As $z\mapsto \frac{\omega^k(z)}{z^2}$ is not integrable, Proposition \ref{zeroP} remains valid when the assumptions are weakened and only assume that $Pf$ (resp. $P^+f$) is integrable for the weighted measure $\omega^{k} dV$  for $k\geq -1$.  
\end{remark}

We are now ready to finish the proof of the necessity conditions of Theorem \ref{Bergman}.  Let $f\in L^1(\C_+)$ be a non negative function outside a compact set of $\C_+$ such that 
 $P^+(f)$ belongs to $L^1(\C_+)$. Without loss of generality, we can assume  that $f$ is non negative for $\zeta\notin B(i, 1/2)$ and has norm in $L^1$ equal to $1.$ We already know that $f$ has integral $0$ so that
 $$P^+f(z)=\int (|K(z, \zeta)|-|K(z, i)|) f(\zeta) dV(\zeta).$$ 
 Moreover, since 
 $\displaystyle \int (|K(z, \zeta)|-|K(z, i)|) f(\zeta) dV(\zeta)$ is in $L^1$ and $f_{-}$ has compact support, we can replace $f$ by $f_{+}. $ So the question becomes: {\sl Let $f\in L^1$ be a non negative function such that $$P^+_{\mod} f=\int (|K(\cdot, \zeta)|-|K(\cdot, i)|) f(\zeta) dV(\zeta)\in L^1(\C_+)$$ then prove that $f$ satisfies \eqref{logB}.}  The proof is given in the next two lemmas
 
 The first step mimics the proof of Lemma \ref{ineq-fund}. We give its proof for completeness, even if the proof is already  given in \cite{BGS1}.
 \begin{lemma}\label{boundary}
Under the assumptions that $f$ is non negative and $P^+_\mod f$ is in $L^1$, we have $$\int_{\C_+ } f(\zeta) \ln_+\left(\frac{1}{\Im m\zeta}\right)dV(\zeta)<\infty.$$
\end{lemma}
\begin{proof}
 If we reduce the integral of $|P_{\mod}^+f(z)|$ to the one over the set $y<1,$ then the integral coming from the second term is bounded since 
 $$\int_{\{x+iy\in \C_+ ; y<1\}}\int_{\C_+} |K(z, i)| f(\zeta) dV(\zeta)dV(z)\lesssim \int_0^1 (y+1)^{-1}dy \leq \ln (2).$$
 Finally by difference, we get that 
 $$\int_{\{x+iy\in \C_+ ; y<1\}} \int_{\C_+} |K(z, \zeta)| f(\zeta) dV(\zeta)dV(z)<\infty.$$
We integrate first in $z$ and find that
$$\int_{\C_+} f(\zeta) \ln\left(1+\frac{1}{\Im m\zeta}\right)dV(\zeta)<\infty.$$
 \end{proof}
 \begin{lemma}\label{far}
Under the previous assumptions on $f$, we have $$\int_{\C_+ } |f(\zeta)| \ln_+(|\zeta|)dV(\zeta)<\infty.$$
\end{lemma}
\begin{proof}
 We only consider the integral of $P^+_{\mod} f$ over the set $$F:=\{z=x+iy ; |x|\leq y, y>1\}.$$ We write for $z\in F$, 
$$P^+_{\mod} f(z):=I_1(z)+I_2(z)+I_3(z),$$
where $$I_j(z):=\int_{E_j}K^+_{\mod}(z, \zeta)f(\zeta) dV(\zeta).$$ We will choose $E_j, j=1,2, 3$ so that $\int_F |I_j(z)|dV(z)\lesssim 1$ for $j=2, 3$, while 
$K^+_{mod}(z, \zeta)$ is non positive on $E_1$. More precisely, we take 
\begin{eqnarray*}
E_1&:=&\left\{|z+i|< \frac 14 |\zeta-i|\right\}\\
E_2&:=&\left\{\frac 14 |\zeta-i|<|z+i|<4|\zeta-i|\right\}\\  E_3&:=&\left\{ |z+i|>4|\zeta-i|\right\}.
\end{eqnarray*}
Under the first condition, the second term of $K^+_{\mod}$ is the largest one, which gives the non positivity. Under the second condition, since $|K^+_{\mod}(z, \zeta)|\lesssim y^{-2},$ its integral in $z$ on a domain of the form $R\leq |z+i|\leq CR$ is uniformly bounded. Finally, in the third domain, we use the fact that the kernel is given by a difference, which  is bounded, up to a constant,  by $|\zeta-i|\times |z+i|^{-3}.$ Its integral in $z$ is uniformly bounded. So these last two terms are bounded under the unique assumption that $f$ is integrable. Finally, under the assumption that $P^+_{\mod} f$ is integrable, we get 
 \begin{equation}\label{crucial}\int_{F}\int_{|z+i|< \frac 14 |\zeta-i|}|z+i|^{-2} f(\zeta)dV(\zeta)dV(z)<\infty.\end{equation}
We integrate first in $z$ to find that
$$\int_{\C_+} f(\zeta) \ln_+(|\zeta-i|)dV(\zeta)<\infty.$$
This finishes the proof.
 \end{proof}

    
  $$
 $$ we have 
 $$
 $$

 \bigskip

Next, we consider weighted spaces with weighted measure $\omega^k dV$
where $\omega$ is defined by \eqref{omega}. We only give  sufficient conditions for $Pf$ to be in $L^1_{\omega^k}$ in the next proposition. We outline the fact that the result depends on the power $k$. Basically, when $k>-1$, the sufficient condition for $Pf$ to be in $L^1_{\omega^k}$ may be written in terms of the weight $\omega^{k+1}$. For $k=-1,$ the condition is given in terms of the $\ln \ln$ function. For $k<-1$ there is no further condition since the kernel is now uniformly integrable for the measure $\omega^k dV$. 
\begin{proposition} \label{BergmanW}
Let $f$ be an integrable function. Sufficient conditions for $Pf$ (resp. $P^+f$)  to be in $L^1_{w^{k}}(\C_+)$ are given in the following enumeration.
\begin{itemize}
\item If $k>-1,$  it is sufficient that  $f$ satisfies the two conditions:
\begin{equation}\label{meanBW}
\int_{\C_+} f dV(z) =0,
\end{equation}
\begin{equation}\label{logBW}
\int_{\C_+} |f(z)|\omega^{k+1}(z) dV(z) <\infty.
\end{equation}
\item If $k=-1$, it is sufficient that $f$ satisfies the condition \eqref{meanBW} and the condition
\begin{equation}\label{logBW1}
\int_{\C_+} |f(z)|(1+\ln_+(\ln_+ (\Im m(z))^{-1}))+\ln_+(\ln_+ (|z|)) dV(z) <\infty.
\end{equation}
\item If $k<-1$, $Pf$ belongs to $L^1_{\omega^k}(\C_+)$ with no extra condition.
\end{itemize}
\end{proposition}
The proof is analogous to the sufficiency in Proposition \ref{Bergman}. In the two first cases, we already know that \eqref{meanBW} is a necessary condition, so that we assume they are satisfied and the kernel can be replaced by the modified kernel. One adapts easily Lemma \ref{ineq-fund}. The main difference is the fact that the function $t\mapsto \int_2^t y^{-1} (\ln y)^{k} dy $ replace now the logarithm. More precisely, we use the  elementary inequality
\begin{equation}
\int_2^t \frac {\ln(s)^k}{s }ds \lesssim\left\{
    \begin{array}{lll}
        c_k & \mbox{if } k<-1, \\
        \ln(\ln(t)  & \mbox{if } k=-1,\\
        c_k(\ln (t))^{1+k } & \mbox{if } k>-1. 
    \end{array} \right. \label{logW}
    \end{equation}
 Details are given in \cite{BGS2} for $k=-1.$ For $k<-1$, the kernel $K(\cdot, \zeta)$ belongs to $L^1_{\omega^k}(\C_+)$ uniformly in $\zeta$ hence $Pf\in L^1_{\omega^k}(\C_+)$ for any $f\in L^1(\C_+).$

\smallskip

\medskip

\section{Bloch functions, multipliers, duality}

Let us recall that, for $k\in \mathbb{R}$, the weighted Bloch space $\mathcal{B}_{\omega^{k}}(\mathbb{C}_+)$ is the space of equivalent classes of holomorphic functions which satisfy
$$\sup_{z\in\C_+} \left(\Im m(z)\right)\omega^{k}(z)|f'(z)|< \infty.$$ As we will consider multipliers on these spaces, we will use functions in $\mathcal B_{\omega^k}$ instead of equivalent classes with the norm
\begin{equation} \label{eq:wbloch}
  \|f\|_{\mathcal B_{\omega^{k}}} :=|f(i)|+\sup_{z\in\C_+} \left(\Im m(z)\right)\omega^{k}(z)|f'(z)|< \infty.
\end{equation} 
Remark that we only defined a semi-norm on $\mathcal B_{\omega^k}$ in formula \eqref{semi}.

We start by observing the following fact.
\begin{lemma}\label{lem:pointwisebloch}
Let $k\in \R$. Then there is a constant $C_k>0$ such that for any $f\in \mathcal{B}_{\omega^{k}}(\mathbb{C}_+)$ of norm $1$,
$$
|f(z)| \leq\left\{
    \begin{array}{lll}
      C_k & \mbox{if } k>1, \\
        C_{-1} \ln(\omega(z))  & \mbox{if } k=1,\\
        C_k (\omega(z))^{1-k } & \mbox{if } k<1.
    \end{array}
\right.$$

\end{lemma}
Remark that we have the equivalence 
$$1+\ln (\omega(z))\simeq 1+ \ln_+(\ln_+(1/y))+\ln_+(\ln_+(|z|))$$
and recognize the quantity that appears in \eqref{logBW1}.

\begin{proof}

We write, for  $|x|\le y$,
$$|f(z)-f(i)|\leq |f(x+iy)-f(iy)|+|f(i y)-f(i)|$$
and, for $|x|\geq y$
\begin{align*}|f(z)-f(i)|\leq & |f(x+iy)-f(x+i|x|)|+|f(x+i|x|)-f(i|x|)|\\&+|f(i|x|)-f(iy)|+|f(iy)-f(i)|.\end{align*}
We bound all differences by integrals of derivatives. For instance, we get for  $|x|\le y$,
\begin{eqnarray*}
|f(z)-f(i)|&\le& \int_0^{|x|}|f'(t+iy)|dt+\int |f'(is)|\chi_{s\in (1,y)}ds \\
&\le& \frac 1y\int_0^{|x|}\omega^{-k}(t+iy)dt+\int \frac{ \omega^{-k}(is)}{s}\chi_{s\in (1,y)} ds. 
\end{eqnarray*}
and use the inequalities \eqref{logW} to conclude. 
\end{proof}
Critical examples are given by $z\mapsto 1$ when $k>1,$ by $z\mapsto\log \log (4i+z)$ when $k=1$, by $z\mapsto (\log(4i+z))^{1-k}$ when $k<1.$ For these last examples we use the fact that $\log (4i+z)$ has a positive real part, so that its logarithm is well defined as a holomorphic function.

These estimates may be used to give two kinds of estimates: first on pointwise products of functions that are respectively in the Bergman space $A_{\omega^l}^1(\C_+)$ and  in the weighted Bloch space $\mathcal{B}_{\omega^{k}}$, secondly for multipliers of the weighted Bloch spaces.

\subsection{Pointwise products}
As a direct consequence, we have the following corollary.
\begin{corollary}\label{emb-weak} 
Let $k<1$, and $l\in\mathbb{R}$. The pointwise product of a function  in the Bergman space $A_{\omega^l}^1(\C_+)$ and a function in the weighted Bloch space $\mathcal{B}_{\omega^{k}}(\mathbb{C}_+)$ belongs to the space $A^1_{\omega^{l+k-1}}(\C_+).$ Moreover, for $f\in A_{\omega^l}^1(\C_+)$ and $g\in\mathcal{B}_{\omega^{k}}(\mathbb{C}_+),$ one has
$$\|fg\|_{A^1_{\omega^{l+k-1}}}\lesssim \|f\|_{A_{\omega^l}^1}\|g\|_{\mathcal{B}_{\omega^{k}}}.$$
\end{corollary}
\medskip

Remark that multiplication by  a function $ b\in\mathcal{B}_{\omega^{k}}(\mathbb{C}_+)$, for $k>1,$ preserves all spaces $A_{\omega^l}^1(\C_+)$ since $b$ is bounded.  One could also write a statement for multiplication by a function in $\mathcal{B}_{\omega}(\mathbb{C}_+)$. It would involve a logarithm.

We will see that the result given here is optimal.

\subsection{Multipliers} 
We start by recalling that a holomorphic function $b$ is a pointwise multiplier of the Bloch space $\mathcal{B}(\mathbb{C}_+)$ if and only if $b$ is bounded and belong to $\mathcal{B}_{\omega}(\mathbb{C}_+)$ (see  \cite{BGS1}). The fact that multipliers satisfy a logarithmic condition at infinity was remarked by Nakai \cite{Nakai} in the case of $BMO.$

We would like to extend this result to all $\mathcal{B}_{\omega^{k}}(\mathbb{C}_+)$, $k>-1$ and give other characterizations of multipliers between different different spaces of this type. For $\Omega$ a positive function on $\mathbb{C}_+$, we introduce the set $H_{\Omega}^{\infty}(\mathbb{C}_+)$ of all holomorphic functions $f$ such that
\begin{equation}\label{eq:hinfty}
\|f\|_{H_{\Omega}^{\infty}}:=\sup_{z\in \mathbb{C}_+}|f(z)|\Omega(z)<\infty.
\end{equation}
We note that endowed with the norm $\|\cdot\|_{H_{\Omega}^{\infty}}$, $H_{\Omega}^{\infty}(\mathbb{C}_+)$ is a Banach space.
\medskip

Our result on the pointwise multipliers between different weight Bloch spaces is as follows.
\begin{proposition}
Let $j\leq k$. The space of pointwise multipliers from the weighted Bloch space $\mathcal{B} _{\omega^{k}}(\mathbb{C}_+)$ to the weighted Bloch space $\mathcal{B}_{\omega^{j}}(\mathbb{C}_+)$ consists of, 
\begin{enumerate}
\item when $k>1$, functions in $\mathcal{B}_{\omega^{j}}(\mathbb{C}_+)$;
\item when $k<1$, functions in the class $\mathcal{B}_{\omega^{j+1-k}}(\mathbb{C}_+)\cap H_{\omega^{j-k}}^{\infty}(\mathbb{C}_+)$;
\item when $k=1$,
 functions in $H_{\omega^{j-1}}^{\infty}(\mathbb{C}_+)$
 satisfying the condition
\begin{equation}\label{eq:loglogmult}
\sup_{z\in\C_+} \left(\Im m(z)\right)\omega^j(z)\left(1+\ln\left(\omega(z)\right)\right)|f'(z)|< \infty.
\end{equation} 

\end{enumerate}
\end{proposition}
\begin{proof}
Let $f\in \mathcal B_{\omega^k}$. We write $(fb)'=f'b+b'f$. Using the definition of Bloch spaces, the proof of the sufficient condition follows easily from Lemma \ref{lem:pointwisebloch} and elementary arguments. 

For the necessary condition, we first observe that, testing on the function $f=1$, the condition $b\in\mathcal B_{\omega^j}$ is necessary. It ends the proof in the case $k>1$. 

So we only need to concentrate on the cases $k<1$ or $k=1$ and points for which $\omega(z)\simeq \ln_+(|z|)$ or $\omega(z)\simeq \ln_+(1/y)$.

Let us explain how we will proceed for the proof. Assuming  that $b$ is multiplier from $\mathcal{B}_{\omega^{k}}(\C_+)$ to $\mathcal{B}_{\omega^{j}}(\C_+)$ , with $k<1$, we first prove that $b$ necessarily belongs to $H_{\omega^{j-k}}^{\infty}(\mathbb{C}_+)$ . From Lemma \ref{lem:pointwisebloch}, we know that there is a constant $C_j>0$ such that for any $f\in \mathcal{B}_{\omega^{k}}(\C_+)$,
 \begin{equation}\label{eq:testpointwise}
 |f(z)b(z)|\le C_j\|f\|_{\mathcal{B}_{\omega^{k}}}\omega(z)^{1-j}.
 \end{equation}
 We need to find  test functions $f \in \mathcal{B}_{\omega^{k}}$ (which may depend on the point $z$) with norm $1$ such that $|f(z)|\geq c_k\omega(z)^{1-k}$. This will  imply  that $\omega^{j-k}(z)|b(z)|\leq c_k^{-1} C_j$. Once we get this bound on $b$, we have $|yb'(z)||f(z)|\leq C \omega(z)^{-j}$ from the definition of a multiplier. So it is sufficient to find  test functions $f\in \mathcal{B}_{\omega^{k}}$ such that $|f(z)|\gtrsim \omega(z)^{-k+1}.$ The construction of such test functions is given by the following lemma.

 \begin{lemma} \label{lem: testTheta}
 We consider the function 
 $$(z, w)\mapsto \theta_w(z):= 1-\log(z-\bar{w})+\ln|i+w|+ 2\log(i+z).$$
 It has the following properties.
 \begin{itemize}
 \item [(i)]
 $\Re e (\theta_w)>1-\ln 2+\ln (|i+z|),$ so that, in particular, its logarithm and its powers are well defined.
 \item [(ii)]
 $$|\theta_w(z)|\simeq 1+\ln_+(|z|) +\ln_+\frac 1{|z-\overline w|}$$
 with constants that do not depend on $w\in\mathbb C_+$.
 \item [(iii)]
   For all $k<1$,  the function $\theta_w^{-k+1}$ belongs to $\mathcal{B}_{\omega^{k}}(\C_+)$ and $$\|\theta_w^{-k+1}\|_{\mathcal{B}_{\omega^{k}}}\simeq 1$$ 
 with constants that do not depend of $w$.
 \item   [(iv)]
 The function $\log(\theta_w)$ belongs to $\mathcal{B}_{\omega}(\C_+)$ and  $$\|\log(\theta_w)\|_{\mathcal{B}_{\omega}}\simeq 1$$
  with constants that do not depend on $w$.
 \end{itemize}
 \end{lemma}
 
{\sl End of the proof of the necessary condition.} Let us assume this lemma proved and finish the proof when $k<1$. We first prove the estimate when $\omega(z)\simeq \ln_+(|z|).$ It is sufficient to take as test function the function $\theta_w$ with $w=i.$ Indeed, with this choice, we know from (ii) that $|\theta_w|\simeq \omega(z)$ and we conclude. Next, assuming that $\omega(z)\simeq \ln_+(1/y),$ $z=x+iy$, we choose $w=z,$ so that $|z-\overline w|=2y.$ Again $|\theta_w|\simeq \omega(z)$ and we conclude in the same way. Once we have the boundeness of $b$, we conclude that $b$ is in $\mathcal B_{\omega}$ by testing in each case on the same functions.

 The proof when $k=1$ is identical, but by taking $\log (\theta_w)$ as test functions, with the same values of the parameter $w$. We do not give the details again.

\end{proof}

 \begin{proof}[of Lemma \ref{lem: testTheta}]
 We first prove assertion (i). From the definition of $\theta_w$, 
 $$\Re(\theta_w(z))=1-\ln |z-\overline w|+\ln|i+w| +2\ln |i+z|.$$
 As, for $z,w\in\C_+$, $$|z-\overline w|\le |i+z|+|i+w|,$$ one gets, from the inequality $(a+b)\le 2 ab$, $a,b\ge 1$, that 
 $$\ln |z-\overline w|< \ln|i+w|+\ln|i+z|+\ln 2$$  so that $\Re(\theta_w(z))\ge  1-\ln 2+\ln|i+z|$.
 
 Let us prove (ii). We first consider the case when $|z-\overline w|\ge 1$. 
Similarly as before, we have
 $$ \ln(|i+w|)\le \ln 2+\ln(|i+z|) +\ln(|z-\overline w|).$$ It gives $$\Re(\theta_w(z))\le 1+\ln(2) +3\ln |i+z|.$$
 From the preceding inequality, we have also
 $$\Re \theta_w(z)\gtrsim 1+\ln_+(|z|)$$
 so that
 $$|\theta_w(z)|\simeq 1+\ln_+(|z|)$$
 which is also equal to 
 $$1+\ln_+(|z|)+\ln_+\frac 1{|z-\overline w|}.$$ 
 
When  $|z-\overline w|\le 1$, $\ln|i+w|\simeq \ln |i+z|$, and 
 \begin{eqnarray*}
 \Re (\theta_w(z)) &=&1+\ln \frac 1{|z-\overline w|}+\ln|i+w| +2\ln |i+z|\\ 
 &\simeq& 1+\ln_+\frac 1{|z-\overline w|}+\ln_+(|z|).
 \end{eqnarray*}
It ends the proof of (ii).
 
Let us now prove (iii), that is,
$$\frac{y}{|z+i|} \left(\frac{|\theta_w(z)|}{\omega(z)}\right)^{-k}+\frac{y}{|z-\overline w|} \left(\frac{|\theta_w(z)|}{\omega(z)}\right)^{-k}\lesssim 1, \forall z=x+iy\in \C_+.$$
 For $k\le 0$ we conclude directly since, from (ii), $|\theta_w(z)|\lesssim \omega(z).$ It remains to consider the case $k>0.$
 The inequality $$ \frac{y}{|z+i|} \lesssim \left(\frac{|\theta_w(z)|}{\omega(z)}\right)^k$$ holds  when $\omega (z)\simeq 1$ or $\omega(z)\simeq\ln_+(|z|)$ since, thanks to (ii), 
  the right hand side is bounded below by a constant.  So, it is sufficient to consider the case when $\omega(z)\simeq \ln_+(1/y).$ Then, as 
 $$y\left(\ln_+(1/y)\right)^{k}\lesssim 1$$
 while both $|z+i|$ and $|\theta_w(z)|$ are bounded below, we get the result.

 We then prove that 
  $$ \frac{y}{|z-\overline w|} \lesssim \left(\frac{|\theta_w(z)|}{\omega(z)}\right)^k .$$ 
  The proof is analogous except when $\omega(z)\simeq \ln_+(1/y).$ In this case we remark that 
  $$y(\ln_+(1/y))^k\leq y(\ln(e+1/y))^k\leq |z-\overline w|\ln(e+1/|z-\overline w|)^k$$
 by monotonicity of the function $t\mapsto t(\ln (e+1/t))^k.$ The right hand side is bounded by $|z-\overline w||\theta_w(z)|^k$, which allows to conclude the proof of (iii).

 The proof of (iv) follows the same lines.
\end{proof}

 \subsection{Reproducing formula for $A^1_{\omega^{k}}$-functions}
 We first recall some notations. For $\Omega$ a weight, the space $L_\Omega^p(\mathbb{C}_+)$ is the $L^p$ space for the measure $\Omega dV$ and $A^p_\Omega(\C_+):= L_\Omega^p(\mathbb{C}_+)\cap \mathcal H(\C_+)$  the corresponding Bergman space. 
When $\Omega(z)=\left(\Im m(z)\right)^\alpha$, we recover the classical (weighted) Bergman space denoted by $A_\alpha^p(\mathbb{C}_+)$. This space is nontrivial only if $\alpha>-1$. Finally, for $\alpha>-1$, the Bergman projection $P_\alpha$ is the orthogonal projection from $L_\alpha^2(\mathbb{C}_+)$ onto its closed subspace $A_\alpha^2(\mathbb{C}_+)$. It is defined by $$P_\alpha f(z)=\int_{\mathbb{C}_+}K_\alpha(z,w)f(w) dV_\alpha(w)$$
where $$K_\alpha(z,w)=\frac{c_\alpha}{\left(z-\bar{w}\right)^{2+\alpha}}$$
is the Bergman kernel, $dV_\alpha(w)=(\Im m(w))^\alpha dV(w)$ and $c_\alpha$ is a constant that depends only on $\alpha$.
\medskip
 
 The aim of this section is to prove the following proposition.
 \begin{proposition}\label{cor:boundednessP1omegainverse}
For $k\in \mathbb{R}$, the Bergman projection $P_\alpha$, $\alpha>0$,  reproduces the functions in $A^1_{\omega^{k}}(\mathbb{C}_+)$.
\end{proposition}
We need preliminary results.
 
Let us start with the following estimates that generalize the ones from \cite{BGS1}, and deal with this kind of kernels. The proof follows the lines of \cite{BGS1}.
 \begin{lemma}\label{lem:weightforellirudin}
Let $\alpha>0$, $\beta>-1$, and $k,s\in \R$. Let $$\Omega(z)=\left(\ln^{\varepsilon_1}(e+|z|)+\ln^{\varepsilon_2}\left(e+\frac{1}{\Im mz}\right)\right)^{ k}\left[\ln\left(e+\ln^{\varepsilon_3}(e+|z|)+\ln^{\varepsilon_4}\left(e+\frac{1}{\Im mz}\right)\right)\right]^{s},$$
$\varepsilon_j\in\{0,1\}$, $j=1,2,3,4$. Then there is a constant $C=C_{\alpha,\beta}>0$ such that for any $z_0\in \mathbb{C}_+$,
\begin{equation}\label{eq:weightforellirudin1}
    \int_{\mathbb{C}_+}\frac{\Omega(z)dV_\beta(z)}{|z-\overline z_0|^{2+\alpha+\beta}}\le C(\Im m(z_0))^{-\alpha}\Omega(z_0).
\end{equation}
\end{lemma}

We also need the following proposition.
\begin{proposition}\label{prop:densityrhoomega}
Let $k>0$. Then
the set $A^1(\mathbb{C}_+)$ is dense in $A^1_{\omega^{-k}}(\mathbb{C}_+)$.
\end{proposition}
We refer to \cite{BGS1} for the proof, which is inspired in particular by \cite{BBGNPR, st-book-fourier-analysis}. It is easily adapted from the statement given there  when $k=1.$ As in this reference, one first consider  the weighted Bergman space $A^1_{\rho^k}(\mathbb{C}_+),$ with $\rho:=(1+\ln(e+1/y))^{-1}$, $k>0$. We leave details of verification to the reader.

\begin{proof}[of Proposition \ref{cor:boundednessP1omegainverse}]
For $k>0$, the result follows from the continuous embedding $A^1_{\omega^k}(\mathbb{C}_+)\hookrightarrow A^1(\mathbb{C}_+)$ and that $P_\alpha$ reproduces functions of $A^1(\mathbb{C}_+)$. 
\medskip

To obtain that $P_\alpha$ also reproduces functions in $A^1_{\omega^{-k}}(\mathbb{C}_+)$, $k>0$, we first observe that $P_\alpha$ maps $L^1_{\rho^k}(\mathbb{C}_+)$ boundedly into $A^1_{\rho^k}(\mathbb{C}_+)$. The conclusion for this case then follows from the latter and Proposition \ref{prop:densityrhoomega}.
\end{proof}

 \subsection{Duality} 
 Let us use the notation $L_{\omega^{-k},\alpha}^1(\mathbb{C}_+)=L^1(\mathbb{C}_+,\omega^{-k}(z)dV_\alpha(z)).$ 
\begin{proposition}
Let $\alpha> 0$, and $\beta>-1$. Then the Bergman projector $P_{\alpha+\beta}$ is bounded from $L_{\omega^{-k},\beta}^1(\mathbb{C}_+)$ to $L_{\omega^{-k},\beta}^1(\mathbb{C}_+)$.
\end{proposition}
\begin{proof}
This is a direct consequence of Fubini's Theorem and Lemma \ref{lem:weightforellirudin}.          
\end{proof}
Let  us define the following weighted space. $$L_{\omega^k,\alpha}^\infty(\mathbb{C}_+):=\{f\,\,\textrm{measurable}:\sup_{z\in\mathbb{C}_+}(\Im m(z))^\alpha\omega^k(z)|f(z)|<\infty\}.$$
Then $L_{\omega^k,\alpha}^\infty(\mathbb{C}_+)$ is a Banach space with norm
$$\|f\|_{\infty,\omega^k,\alpha}:=\sup_{z\in\mathbb{C}_+}(\Im m(z))^\alpha\omega^k(z)|f(z)|.$$
The following is elementary.
\begin{proposition}\label{prop:inftyboundedness}
Let $\alpha>0$, and $\beta>-1$. Then the Bergman projection $P_{\alpha+\beta}$ maps $L_{\omega^k,\alpha}^\infty(\mathbb{C}_+)$ boundedly into itself.
\end{proposition}
 We have the following duality result. 
 
\begin{theorem}\label{thm:dualweighted} Let $k\in \R.$
The dual space $\left(A_{\omega^{-k}}^1(\mathbb{C}_+)\right)^*$ of $A_{\omega^{-k}}^1(\mathbb{C}_+)$ identifies with $\mathcal{B}_{\omega^{k}}(\mathbb{C}_+)$ modulo constants under the pairing
\begin{equation}\label{eq:weighteddualpairing}
   \left<f,g\right>_{*}=\int_{\mathbb{C}_+}f(z)\overline{g'(z)}(\Im m(z))dV(z), 
\end{equation}
$f\in A_{\omega^{-k}}^1(\mathbb{C}_+)$, and $g\in \mathcal{B}_{\omega^{k}}(\mathbb{C}_+)$.
\end{theorem}
\begin{proof}
       Let $g\in \mathcal{B}_{\omega^{k}}(\mathbb{C}_+)$. Then that (\ref{eq:weighteddualpairing}) defines a bounded linear operator is direct from the definition of the spaces involved.  
       \medskip
       
       Conversely, assume that $\Lambda$ is an element of $\left(A_{\omega^{-k}}^1(\mathbb{C}_+)\right)^*$. Then $\Lambda$ can be extended as an element $\tilde{\Lambda}$ of $\left(L_{\omega^{-k}}^1(\mathbb{C}_+)\right)^*$ with the same operator norm. Then classical arguments give that there is an element $h\in L^\infty(\mathbb{C}_+)$ such that for any $f\in L_{\omega^{-k}}^1(\mathbb{C}_+)$,
$$\tilde{\Lambda}(f)=\int_{\mathbb{C}_+}f(z)\overline{h(z)}\omega^{-k}(z)dV(z).$$
In particular, we have that for any $f\in A_{\omega^{-k}}^1(\mathbb{C}_+)$,
$${\Lambda}(f)=\int_{\mathbb{C}_+}f(z)\overline{h(z)}\omega^{-k}(z)dV(z).$$

By Proposition \ref{cor:boundednessP1omegainverse}, $P_1$ reproduces functions in $A_{\omega^{-k}}^1(\mathbb{C}_+)$. Hence, for $f\in A_{\omega^{-k}}^1(\mathbb{C}_+)$,
\begin{eqnarray*}
{\Lambda}(f) &=&\int_{\mathbb{C}_+}f(z)\overline{h(z)}\omega^{-k}(z)dV(z)\\
&=& \int_{\mathbb{C}_+}P_1(f)(z)\overline{h(z)}\omega^{-k}(z)dV(z)\\ &=&
\int_{\mathbb{C}_+}f(z)\overline{P_1g(z)}(\Im m(z))dV(z)
\end{eqnarray*}
where $g(z)=\omega^{-k}(z)(\Im m(z))^{-1}h(z)$. Here, we applied Fubini's Theorem to obtain the last equality.\\ From Proposition \ref{prop:inftyboundedness}, $P_1$ is bounded on $L_{\omega^{k},1}^\infty(\mathbb{C}_+)$. As $g$ belongs to  $L_{\omega^{k},1}^\infty(\mathbb{C}_+)$, the same holds for $P_1(g)$. It follows that if we define $G$ to be a solution of $G'(z)=P_1g(z)$, then $G\in \mathcal{B}_{\omega^{k}}(\mathbb{C}_+)$. Hence $${\Lambda}(f)=\int_{\mathbb{C}_+}f(z)\overline{G'(z)}(\Im m(z))dV(z).$$
That is, $\Lambda$ is given by (\ref{eq:weighteddualpairing}). The proof is complete.
\end{proof}
   

\section{Products of functions in $A_{\omega^l}^1$ and $\mathcal{B}_{\omega^{k}}(\mathbb{C}_+)$ and Hankel operators on $A^1(\C_+)$}
\subsection{Factorization of an atom}

As stated in corollary \ref{emb-weak}, the product of a function in $A^1_{\omega^l}(\C_+)$ and a function in $\mathcal B_{\omega^k}(\C_+)$ belongs to $A^1_{\omega^{l+k-1}}(\C_+)$. The purpose of this paragraph is to consider the converse statement in a particular case, namely we will factorize an atom of $A^1_{\omega^{l+k-1}}(\C_+)$.
We consider the following function.
\begin{equation}\label{eq:atom}
    f(z)=\frac{(\Im m (w))\omega^{1-(l+k)} (w)}{(z-\overline w)^3}.
\end{equation} 
Using Lemma \ref{lem:weightforellirudin}, one obtains $\|f\|_{L^1(\omega^{l+k-1})}\simeq 1$. For this weighted kernel, we actually have a strong factorization as shown below. This result will be useful in proving the necessary condition in the boundedness of Hankel operators on $A^1(\C_+)$.
\begin{proposition}\label{prop:factor}
Let $f$ be an atom given by (\ref{eq:atom}). There exist $g\in A_{\omega^{l}}^1(\C_+)$ and $\theta \in \mathcal{B}_{\omega^{k}}(\mathbb{C}_+)$, $k<1$ such that $$f=g\times \theta$$
with 
$$\Vert g\Vert_{A_{\omega^{l}}^1}\times \Vert \theta\Vert_{\mathcal B_{\omega^{k}}}\lesssim 1.$$
\end{proposition}
\begin{proof}
We choose $\theta=\theta_w^{-k+1}$ where the function $\theta_w$ is defined in Lemma \ref{lem: testTheta}. In this case, the function $g$ is given as follows.

$$g:z\mapsto \frac{(\Im m (w))\omega^{1-(l+k)} (w)}{(z-\overline w)^3\theta(z)}$$ which is a well defined holomorphic function in the upper half plane. It remains to prove that the inequality $\|g\|_{L_{\omega^l}^1}\lesssim 1$ holds. When $\omega(w)\simeq 1,$ then $\theta\simeq 1$ and there is nothing to prove.  We will now use the two notations
$\omega_1(z)=\ln\left(e+\frac{1}{\Im mz}\right)$ and $\omega_2(z)=\ln(e+|z|)$ and use Lemma \ref{lem:weightforellirudin} for these weights. We distinguish two cases.
\begin{itemize}
    \item [(i)] Case $v<1/2$ and $\ln(v^{-1})\simeq \omega(w)\simeq \omega_1(w)$  We proceed as in \cite{BGS1}. We
cover the upper half plane by the union of $E_j'$s, with $E_0:=Q_w$ the Carleson square centered at $w$, and $E_j:=Q_{w_{j}}\setminus Q_{w_{j-1}}, j\geq 1.$ Here $w=u+iv$ and $w_j=u+i2^j v$.

We recall that for $z\in E_j,$ we have $|z-\overline   w|\leq 2\Im m(w_j)$. Hence using Lemma \ref{lem: testTheta}, we obtain
\begin{eqnarray*}
     \left|\theta(z)\right|&\gtrsim& \left(1-\ln 2+\ln_+\left(\frac{1}{|z-\overline w|} \right)\right)^{-k+1}\\
    &\gtrsim& \left(\ln\left(e+\frac 1{\Im m(w_j)} \right)\right)^{-k+1} =\omega_1^{-k+1}(w_j).
\end{eqnarray*}  
Lemma  \ref{lem:weightforellirudin} allows to get
$$\int|g(z)|\omega^l(z)dV(z)\lesssim\sum_j \frac{\omega_1^{1-(l+k)}(w)}{\omega_1^{1-(l+k)}(w_j)}2^{-2j}.$$
To conclude, in the case $1-(l+k)>0$, we use the fact that $\omega_1(w_j)\geq\omega_1(w)-j\ln 2$, and in the case $1-(l+k)\leq 0$, the fact that $\omega_1(w)\geq \omega_1(w_j)$.
\item [(ii)]\; Case $\omega(w)\simeq \ln(e+|w|)=\omega_2(w)$.\\ As $|\theta(z)|\gtrsim \left(\ln(e+|z|)\right)^{-k+1}=\omega_2^{-k+1}(z)$, the estimate follows as in the previous case.
\end{itemize}
\end{proof}

This factorization of atoms can be used to have a weak factorization of the space $A^1_{\omega^{l+k-1}}(\C_+)$ as well, as in \cite{BGS1}.
\medskip

\subsection{Hankel operators}\label{Hankel}

We only consider the unweighted situation in this section.

The Hankel operator $h_b$ is  defined  by $h_b(f):=P(b\overline f)$ when this makes sense. In particular, this is the case when $b$ is in $ \mathcal B$ and  $f$ in $A^1_\omega(\C_+)$. Let us recall how one proves that  $h_b$ is bounded on $A^2(\C_+)$ for $b\in \mathcal B$.  Assuming that $g\in A^2(\C_+)$ belongs to the dense subspace for which $\omega^2|g|^2$ is integrable, we have
$$\langle h_b(f),g\rangle =\langle b,fg\rangle, \; f\in A^2(\C_+)\cap A^1_\omega(\C_+). $$ 
As the product of two $A^2(\C_+)$ functions belongs to $A^1(\C_+)$,   this quantity is well defined when $g\in A^2(\mathbb C_+)$ and $b$ belongs to the dual space of $A^1(\C_+)$ namely for $b\in\mathcal B$, and we have the inequality
$$\|h_b(f)\|_{A^2}\leq \|b\|_{\mathcal B} \|f\|_{A^2}.$$
Since $h_b$ satisfies a continuity inequality on a dense set, it extends  to $A^2(\C_+)$ as a bounded antilinear operator of $ A^2(\C_+)$ into itself, which finally may be defined by the duality bracket
\begin{equation}\label{defHb}
  \langle h_b(f),g\rangle =\langle b,fg\rangle_{A^1,\mathcal B}   
\end{equation}   for any $g\in A^2(\C_+).$ This is the usual definition of the Hankel operator on $A^2(\C_+)$. \medskip

We would like to mimic  this proof in order to  consider Hankel operators from $A^1(\C_+)$ to $A^1(\C_+)$. 
The next lemma proves that there is an obstruction.
\begin{lemma}
Assume that $b$ is in the Bloch class, $f\in A^1(\C_+)$ and $h_b(f)$ is integrable. If moreover $bf$ is integrable, then $f$ satisfies the identity  $\langle f,b\rangle_{A^1,\mathcal B} =0$.    
\end{lemma}
\begin{proof}
Just use Proposition \ref{zeroP}. It is in particular the case if $b$ is bounded or if $f$ belongs to $A^1_{\omega}(\C_+)$. 
\end{proof}


 Hence, if one wants to define Hankel operators $h_b$ as an operator acting from $A^1(\C_+)$ into itself, it means that we restrict ourselves to the subspace of $A^1(\C_+)$ given by  functions $f$ in $A^1(\C_+)$ with $\langle b,f\rangle_{A^1,\mathcal B} =0$. We proceed as we did before and consider first  the operator $h_b^{\rm {mod}}$, which is defined by using the operator $P^{\rm {mod}}$ in place of $P$. There is no restriction for this modified operator, so that we can use the density of  $A^1(\C_+)\cap A^2(\C_+)$. We have, by Fubini's Theorem,
 \begin{equation}\label {mod1}
     \langle h_b^{\rm {mod}} f, g \rangle=\langle b, (g-g(i)) f\rangle,\; f\in A^1(\C_+)\cap A^2(\C_+)
 \end{equation} when $g-g(i)$ belongs to $A^2(\C_+)$.
We now prove the following result.
\begin{proposition}
    Assume that $b\in\mathcal B_\omega$. The following holds:
    \begin{enumerate}
    \item
         The operator $h_b^{\rm {mod}}$ which is defined by 
    \begin{equation}\label {mod2}
        \langle h_b^{\rm {mod}} f, g \rangle=\langle b, (g-g(i)) f\rangle,\; f\in A^1(\C_+)\cap A^2(\C_+)
    \end{equation} 
   extends continuously from $A^1(\C_+)$ into $A^1(\C_+)$.   
  \item  For any $\varepsilon>0$, the operator $h_b$ extends continuously to a bounded operator from $A^1(\C_+)$ into $A^1_{\omega^{-1-\varepsilon}}(\C_+)$.
        \item This extension of $h_b$ maps continuously  $A^1_b(\C_+)$ into $A^1(\C_+)$.
         
    \end{enumerate}
    In the three cases, the operator norm of $h_b$ has norm bounded by $c\|b\|_{\dot {\mathcal B_\omega}}$.  
    \end{proposition}
    \begin{proof}
We first prove the continuity of  the operator defined by \eqref{mod2}.  By the multiplier theorem, for any $g\in \mathcal B$ and $f\in A^1(\C_+)$, $(g-g(i))f$ belongs to $A^1_{\omega^{-1}}$, with the inequality
$$\|(g-g(i))f\|_{A^1_{\omega^{-1}}}\leq C \|f\|_{A^1} |\|g\|_{\dot{\mathcal B}}.$$
We emphasize here that the estimate does not depend on the choice of $g$ in the class $\mathcal B$ since it is written in terms of the difference $g-g(i)$.\\
Hence, for any $b\in\mathcal B_\omega$, any $f\in A^1(\C_+)$, any $g\in \mathcal B$
$$|\langle b, (g-g(i)) f\rangle_{A^1_{\omega^{-1}},\mathcal B_\omega}|\le \Vert b\Vert_{\mathcal B_\omega} \|f\|_{A^1} |\|g\|_{\dot{\mathcal B}}.$$ 
Now, we want to prove that the operator we have defined by \eqref{mod2} is an extension of $h_b$ as defined in \eqref{defHb}. Remark that, for $b\in \mathcal B$, $f\in A^1(\C_+)\cap A^2(\C_+)$ and $g\in A^2(\C^+),$ then \eqref{mod1} may be written as well as 
$$\langle h_b^{\mod} (f), g \rangle=\langle h_b (f), g\rangle -g(i)\langle b, f\rangle,$$
the two terms making sense separately. It remains to prove that the function $g\in \mathcal B$ can be weakly approached by functions in $A^2(\C_+)$. Just take  
$$g_\varepsilon(z):= \frac{g(z+i\varepsilon)}{(1-i\varepsilon z)^N},\; \varepsilon>0.$$
We have proved the first item.


Next, we prove the existence of a constant $C$ such that, whenever $f\in A^1(\C_+)\cap A^2(\C_+)$,
$$\|h_b f\|_{A^1_{\omega^{-1-\varepsilon}}}\leq C \|f\|_{A^1}.$$
For $f\in A^1(\C_+)\cap A^2(\C_+)$, we can write $h_b(f)$ in terms of $h_b^{\rm {mod}} f$ 
$$h_b^{\rm {mod}} f=h_b(f)-\frac 1{(z+i)^2}\langle b,f\rangle.$$ As $h_b^{\rm {mod}} f$ belongs to $A^1(\C_+)$, it is sufficient to prove that the second term is in $A^1_{\omega^{-1-\varepsilon}}(\C_+)$. Hence, we only need to prove that $z\mapsto (z+i)^{-2}$ is in $A^1_{\omega^{-1-\varepsilon}}(\C_+)$. This last assertion follows from the integrability at infinity of $r\mapsto \frac 1{r\ln (1+r)^{1+\varepsilon}}$ for any $\varepsilon>0$. The a priori estimate follows and allows to extend continuously $h_b$ as an operator bounded from $A^1(\C_+)$ into $A^1_{\omega^{-1-\varepsilon}}(\C_+)$

 Finally, if $f$ is in $A_b^1(\C_+),$ then $ h_b f$ and $ h_b^{\rm {mod}} f$ coincide, which allows to conclude for the third statement of the proposition. \\

\end{proof}

One has the reverse property. Namely
\begin{proposition}
    Assume that $b\in\mathcal B$ and $h_b^{\rm {mod}}$ maps continuously $A^1(\C_+)$ into $A^1(\C_+)$. Then $b$ is in $\mathcal B_\omega$ and 
    $$\|b\|_{\dot{\mathcal B}_\omega}\leq C(\|h_b^{\rm {mod}}\|_{A^1\mapsto A^1}+ \|b\|_{\dot{\mathcal B}}).
$$
Moreover, it is sufficient that $h_b$ maps continuously $A_b^1(\C_+)$ into $A^1(\C_+)$ to imply that $b\in \mathcal B_\omega$.
  \end{proposition}
\begin{proof}
    We want to prove that there exists some uniform constant $C$ such that 
    $$|b'(\zeta)| \leq C (\omega(\zeta) \Im m\,\zeta )^{-1}).$$
 From the Cauchy formula, as $z\mapsto  \frac {b(z)}{(\zeta-\overline z)^3}$ is integrable on $\C_+$, we have  $$|b'(\zeta)| =\frac {1 }{\pi} \left|\int_{\C_+} \frac {b(z)}{(\zeta-\overline z)^3} dV(z)\right|.$$
  We set
    \begin{equation*}\label{atom}
    f(z)=\frac{\Im m (\zeta)\omega (\zeta)}{\pi (z-\overline \zeta)^3},
\end{equation*} so that $\langle b,f\rangle=  \Im m (\zeta) \omega(\zeta) b'(\zeta).$ Our aim is to bound this quantity uniformly.
By Proposition \ref{prop:factor}, we can write $$f=\frac{f}{\theta}\times \theta,$$ with
  \begin{equation}\label{sim}
      \left\|\frac{f}{\theta}\right\|_{A^1}\simeq 1,\qquad \qquad \|\theta\|_{\mathcal B}\simeq 1.
  \end{equation}
   Observe that 
  \begin{equation*}
      \left|\langle b, \frac{f}{\theta}\rangle\right|=\left|\int_{\C_+}  \Im m(\zeta) b'(\zeta)\overline{\left(\frac{f}{\theta}\right)}(\zeta) \;dV(\zeta)\right|\lesssim  \|b\|_{\mathcal B}.
  \end{equation*} 
  Hence, we can write
  $$\langle b,f\rangle =\langle h_b^{\rm mod}\frac f\theta,\theta\rangle+\theta(i)\langle b,\frac f\theta\rangle.$$
  From this equality and our assumption we obtain
  $$|\langle b,f\rangle |\lesssim \|h_b^{\rm {mod}}\|_{A^1\mapsto A^1}+\Vert b\Vert_{\dot{\mathcal B}}.$$
  It gives the first result.

  For the second one, we just assume $h_b$ bounded from $A^1_b(\C_+)$ into $A^1(\C_+)$. We will give a bound of $\|h_b^{\rm {mod}}\|_{A^1\mapsto A^1}$ in terms of $\|h_b \|_{A_b^1\mapsto A^1}$. For this we write $f\in A^1(\C_+)$ as a sum $f=g+\langle b, f\rangle u, $ with $u$ some given function such that $\langle b, u\rangle=1$ and that depends only on $b$. We will show later on that such a function exists. For $g$, we use that the two operators coincide, so that
  \begin{eqnarray*}
   \|h_b^{\rm {mod}}(f)\|_{A^1}&\leq & \|h_b\|_{A_b^1\mapsto A^1}\|g\|_{A^1}+ \Vert b\Vert_{\dot{\mathcal B}}\|f\|_{A^1}\|h_b^{\rm {mod}}(u)\|_{A^1}\\
   &\leq & \|h_b\|_{A_b^1\mapsto A^1}\|f\|_{A^1}+\Vert b\Vert_{\dot{\mathcal B}}(\|u\|_{A^1}+ \|h_b^{\rm {mod}}(u)\|_{A^1})\|f\|_{A^1}.
  \end{eqnarray*}
  So there exists two constants $C, C'$ depending only on $b$ such that $$\|h_b^{\rm {mod}}\|_{A^1\mapsto A^1}\leq C \|h_b\|_{A_b^1\mapsto A^1} +C'.$$
  It remains to prove that such a function $u$ exists, assuming that $b$ is not constant.
  We choose $\zeta_0$ such that $|\zeta_0-i|<1/2$ and $b'(\zeta_0)\neq 0$ and take
  $$u(z)=\frac 1{\pi b'(\zeta_0)} (z-\overline{\zeta_0})^{-3}$$ so that $\langle b,u\rangle= 1.$
  Lemma \ref{ineq-fund} and Lemma \ref{lem:weightforellirudin} allow to see that $u$ and $h_b^{\rm {mod}} u$ are integrable. Indeed, for $h_b^{\rm {mod}} u$, we use the fact that $b$ is bounded by $\Vert b\Vert_{\mathcal B_\omega}$ and that the integral of $K^{mod}$ is also bounded by $C\omega$. Then, the result follows from the integrability of $\omega^2 u$.
  
  \end{proof}
  
  \smallskip
  
We have proved Theorem \ref{hankelexp}.   This method does not give the dependence in $b$ of the norm of the operator on $A_b^1(\C_+)$.

\bibliographystyle{plain}

\end{document}